\documentclass[11pt]{amsart}

\usepackage{amscd,amssymb,amsmath,graphicx,mathrsfs,color,url}
\usepackage[all]{xy}

\usepackage[TS1,OT1,T1]{fontenc}

\newtheorem{theorem}{Theorem}[section]
\newtheorem{lemma}[theorem]{Lemma}

\newtheorem{proposition}[theorem]{Proposition}
\newtheorem{conjecture}[theorem]{Conjecture}

\theoremstyle{definition}
\newtheorem{definition}[theorem]{Definition}
\newtheorem{example}[theorem]{Example}

\theoremstyle{remark}
\newtheorem{remark}[theorem]{Remark}

\newcommand{\Z}{\mathbb Z}

\renewcommand{\O}{\mathcal{O}}

\newcommand{\ca}{\mathcal}

\renewcommand{\L}{\mathbb{L}}
\newcommand{\T}{\mathbb{T}}

\newcommand{\lotimes}{\stackrel{\L}{\otimes}}

\newcommand{\bb}{\mathbb}
\newcommand{\hfiber}{\stackrel{h}{\times}}

\begin{document}

\title[Derived Marsden-Weinstein Quotient]{The Derived Marsden-Weinstein Quotient is symplectic}
\author{Jeremy Pecharich}
\address{Mt. Holyoke College, South Hadley, MA}
\email{jpechari@mtholyoke.edu}
\label{I1}\label{I2}\label{I3}

\begin{abstract}
Let $(X,\omega_X)$ be a derived scheme with a $0$-symplectic form and suppose there is a Hamiltonian $G$-action with a moment map for $G$ a reductive group. We prove, under no further assumptions, that symplectic reduction along any coadjoint orbit in the category of derived Artin stacks has a $0$-symplectic form.

\end{abstract}
\maketitle

\section{Introduction}
Symplectic reduction has a vast history going back to the work of Lagrange, Poisson, and Jacobi where it was used  to reduce the dimension of a phase space with an abelian symmetry group. In the 1970's S. Smale generalized Jacobi's method to the case of $SO(3)$. Finally, J. Marsden and A. Weinstein, using the moment map, formulated what is known as symplectic reduction.   Symplectic reduction is now used in an array of contexts such as moduli theory, representation theory, and mathematical physics to name a few.  

We now recall the formulation of symplectic reduction. Let $G$ be a reductive group acting on a symplectic variety $(X,\omega)$ by symplectomorphisms. In addition we assume that the action is Hamiltonian with moment map $\textbf J:X\to \frak g^*$. Since the action is Hamiltonian the level set of a regular value $\mu$ of $\textbf{J}$ is a smooth $G$-invariant subvariety of $X$. The technique of symplectic reduction due to J. Marsden and A. Weinstein then constructs a symplectic form on the quotient space 
$$
\textbf{J}^{-1}(\mu)/G,
$$
if the action is free and proper. 

One can weaken the assumption so that $G$ acts only locally free at the expense of the quotient being a symplectic orbifold. There are many ways to weaken these assumptions further so that the quotient is no longer an orbifold but more generally an Artin stack. When $\mu$ is not a regular value the level set is not smooth and the dimension is different on the various orbits of the action of $G$. However, a deep result of R. Sjamaar and E. Lerman in the $C^{\infty}$ setting is that the quotient space is a union of symplectic manifolds or to be more precise a \emph{stratified symplectic space} \cite{SjaLer}. The jump in the dimension of the quotient is a precursor that the quotient should be taken in the sense of derived geometry \cite{HAGII}. 
The main theorem is the following derived version of the Marsden-Weinstein quotient \cite{MarWei}.
\begin{theorem}[Derived Marsden-Weinstein Quotient] 
\label{derivedMW}
Let $(X,\omega_X,G,\bf J)$ be a Hamiltonian G-space for $G$ a reductive group. Let $j:X\stackrel{h}{\times}_{\frak g^*}\bb O\hookrightarrow X\times \bb O$ be the inclusion map, where $\bb O\subset \frak g^*$ is any coadjoint orbit. Then 
\begin{itemize}
\item the orbit space $\ca X_{red}:=(X\stackrel{h}{\times}_{\frak g^*}\mathbb{O})/G$ is a smooth derived Artin stack,
\item there is a degree $0$ symplectic form $\omega_{red}$ on $\ca X_{red}$ such that $j^*\omega_X=\pi^*\omega_{red}$, where $\pi:X\stackrel{h}{\times}_{\frak g^*}\bb O\to \ca X_{red}$ is the quotient map.
\end{itemize}
\end{theorem} 

Part of the motivation for this generality comes from the BRST formalism, which is an approach to classical field theory in the presence of gauge symmetries. It is the author's hope that one can use the symplectic form constructed in Theorem \ref{derivedMW} to put BRST quantization on a rigorous footing \cite{Costello}. We also hope this will give insight into the philosophy that `quantization commutes with reduction' in the singular case. More precisely, we expect the following conjecture to hold.
\begin{conjecture}[Categorical quantization commutes with reduction] Let $\frak D^G_{qcoh}(X)$ be a dg-enhancement of the derived category of $G$-equivariant quasi-coherent sheaves on $X$ and $\frak D_{qcoh}(\ca X_{red})$ be a dg-enhancement of the derived category quasi-coherent sheaves on $\ca X_{red}$, then there exists an equivalence of dg-categories
$$
\frak D^G_{qcoh}(X)_{\omega_X}\simeq \frak D_{qcoh}(\ca X_{red})_{\omega_{red}},
$$
where the subscript denotes the category deformed in the direction of $\omega_X$, $\omega_{red}$.

\end{conjecture}
\noindent
But quantization in derived geometry is still awaiting a rigorous definition so we will leave these points for future work.

An outline of the paper is as follows. In section 2 we define the needed material from derived stacks after \cite{HAGII}. However, we strongly suggest the survey articles of B. To\"en and G. Vezzosi \cite{Toen2005, HAG2DAG}. We also discuss differential forms on derived stacks due to  D. Ben-Zvi and D. Nadler \cite{BenNad} and derived symplectic forms due to T. Pantev, B. To\"en, M. Vaqui\'e, and G. Vezzosi \cite{PTVV}. The last section proves the main theorem.   Quite surprisingly the proof is very similar to the classical proof of Marsden and Weinstein, cf. Remark \ref{Compare}. This result was also obtained by T. Nevins in the case when the level set is a complete intersection but techniques from derived geometry were not needed in this case \cite{Nevins}.

\emph{Acknowledgements:} I would like to thank Tony Pantev for proposing the problem and many useful discussions. I would also like to thank Ivan Mirkovi\'c for answering many basic questions about derived geometry.

\section{Derived Stacks}

\subsection{Basic notions} In derived algebraic geometry one studies homotopical generalizations of commutative rings. For our purposes we will work over the field of complex numbers, but most of what is said can be generalized to a more general ring. The homotopical replacement that we use is that of non-positively graded commutative differential graded algebras with differential of degree $1$. We denote the category of such algebras with the usual model structure where equivalences are quasi-isomorphisms and fibrations are epimorphisms in degree $\leq -1$ by $cdga_k$.

To be slightly colloquial, the book of algebraic geometry is written in the language of commutative algebras, while the book of derived algebraic geometry is written in the language of commutative dg-algebras. More precisely, we recall that a scheme $X$ is determined by its functor of points i.e.,
$$
X:Alg_k\to Sets,
$$
where $Alg_k$ is the category of commutative algebras and $Sets$ is the category of sets. Using this point of view a \emph{derived scheme} $X$ is a functor
\begin{equation}
\label{dScheme}
X:cdga_k\to Sets.
\end{equation}
We denote the category of schemes and derived schemes by $Sch_k$ and $DSch_k$, respectively.

The category $Alg_k$ is a full subcategory of $cdga_k$ where we view a commutative algebra as a commutative dg algebra concentrated in degree $0$. This gives a pair of adjoint functors
\begin{equation}
\label{adjoint}
\iota:Sch_k\rightleftarrows DSch_k:h^0,
\end{equation}
where $\iota$ places a commutative algebra in degree $0$ and $h^0$ restricts the domain of the functor \eqref{dScheme} to $Alg_k$. It should be noted that for a scheme $X$ the operations on $\iota(X)$ will generally not be equivalent to those in the usual scheme setting. The following example shows this point, and which also serves as our primary example.

\begin{example}[Derived fiber product]
\label{fiberproduct}
Let $A,B,C$ be commutative algebras such that there exists morphisms $B\leftarrow A\rightarrow C$. The \emph{derived fiber product} is the homotopy limit of diagrams $B\leftarrow A\rightarrow C$, which we denote by $B\stackrel{h}{\times}_AC$. Considered as a derived scheme the derived fiber product is the spectrum of the derived tensor algebra
$$
B\lotimes_AC\in cdga_k.
$$
In general, this will not be quasi-isomorphic to the tensor product $B\otimes_AC\in Alg_k$. However, $h^0(B\lotimes_AC)=B\otimes_AC$. If $V,W,X$ are the associated schemes then we denote the derived fiber product by $V\stackrel{h}{\times}_XW$. These derived schemes form a Cartesian square
$$
\xymatrix{V\hfiber_XW\ar[r]^-{p_W}\ar[d]_{p_V}&W\ar[d]\\
V\ar[r]&X.}
$$ 
It follows that 
$$
h^{-i}\left(\O_{V\stackrel{h}{\times}_XW}\right)=Tor_i^X(V,W).
$$
When $V,W\subset X$ are subvarieties that intersect transversally the derived fiber product is quasi-isomorphic to the fiber product in the category of schemes. In some sense the derived fiber product of varieties encodes the non-transversal intersection.
\end{example}

Let $sSet$ be the category of simplical sets then a stack is a functor
$$
X:Alg_k\to sSet.
$$
Following the philosophy for derived schemes from above a derived stack is a functor 
$$
X:cdga_k\to sSet.
$$
However, the intricate details of derived stacks will not be used in the sequel, and we instead refer to the surveys of B. To\"{e}n and G. Vezzosi \cite{Toen2005, HAG2DAG}.

For $X$ a derived Artin stack there is an associated cotangent complex $\bb L_X$. The complex $\bb L_X$ is concentrated in degree $(-\infty,1]$ and is actually perfect, if we assume that $X$ is locally of finite presentation (which will usually be the case) \cite{HAGII}. In particular, the cotangent complex can be dualized to the linear stack $\bb T_X$ concentrated in degree $[1,\infty)$ i.e., $(\bb L_X)^\vee=\bb T_X$. In the case when $X$ is a scheme or an algebraic stack the cotangent complex $\bb L_X$ reduces to that defined by Illusie \cite{Ill}. We illustrate the cotangent and tangent complexes in our primary example.

\begin{example}
\label{cotangentfiber}
Let $V,W,X$ be smooth varieties as in Example 2.1, then the cotangent complex of the derived fiber product is the complex concentrated in degrees $[-1,0]$
$$
\L_{V\hfiber_XW}=[q^*\Omega_X\to p_V^*\Omega_V\oplus p_W^*\Omega_W],
$$
where $q:V\hfiber_XW\to V\to X$. The tangent complex is then 
$$
\T_{V\hfiber_XW}=[p_V^*T_V\oplus p_W^*T_W\to q^*T_X],
$$
where the morphism is given by $p_V^*(df)+p_W^*(dg)$.
\end{example}

\subsection{Differential forms on derived stacks} Our discussion here will be to collect definitions and fix the notation that we use in the proof of the main theorem. The necessary technicalities have been worked out in the excellent paper of D. Ben-Zvi and D. Nadler \cite{BenNad}. 

The basic idea in defining differential forms on a derived stack comes from the homological Hochschild-Kostant-Rosenberg isomorphism for schemes
$$
\O_\Delta\stackrel{\bb L}{\otimes}_{\O_{X\times X}}\O_{\Delta}=\Omega^{-\bullet}_X.
$$
Using algebraic topology the left hand side is the free loop space $\ca LX:=Map (S^1,X)$ of $X$, but with the derived tensor product since the intersection is not transverse.  More precisely, the classifying stack of the affine line $B\bb G_a$ with the additive group action is an affinization of the simplicial space $S^1=B\Z$. Using algebraic geometry the right hand side is functions on the shifted tangent bundle
$$
T_X[-1]=Spec\,Sym^\bullet_{\O_X}(\Omega_X^1[1]).
$$
By a theorem of D. Ben-Zvi and D. Nadler these two interpretations are the same in derived geometry.

\begin{proposition}\cite[Proposition 1.1]{BenNad}
\label{MapTangent}
Fix a derived scheme $X$ and let $\bb L_X$ be the cotangent complex then the derived mapping stack $Map(B\bb G_a,X)$ is equivalent to the odd tangent bundle $\bb T_X[-1]=Spec\, Sym_{\O_X}^\bullet(\bb L_X[1])$. Moreover, there is a canonical equivalence $\bb T_X[-1]\simeq Map (S^1,X)$ given by the affinization $S^1\to B \bb G_a$.
\end{proposition}

On the loop space of a topological space there is a natural $S^1$ action given by rotating the loop. The $S^1$-equivariant cohomology of the free loop space is isomorphic to the de Rham cohomology of $X$ \cite{Loday}. This warrants to try to interpret the de Rham cohomology of a derived stack as equivariant functions on the mapping stack.The action map $B\bb G_a\times X\to X$ is equivalent to giving a vector field of degree $-1$. The associativity of this action implies that the vector field squares to $0$. Using the identification of $\ca LX$ with $\bb T_X[-1]$ from Proposition \ref{MapTangent} the $S^1$-action on $\ca LX$ corresponds to the fiber wise translation given by the $B\bb G_a$-action on $\bb T_X[-1]$. Putting these two points together we have the following.

\begin{proposition}\cite[Proposition 1.2]{BenNad}
Let $X$ be a derived scheme then the translation $B\bb G_a$-action on $\bb T_X[-1]$ is given by the de Rham differential.
\end{proposition}

To recover the grading there is a natural action of $\bb G_m$ on $\bb G_a$ by multiplication which induces an action of $\bb G_m$ on $Map(B\bb G_a,X)$. The equivalence in Proposition \ref{MapTangent} identifies this action with dilation of the fibers of $\bb T_X[-1]$. Moreover, this action recovers the grading on $Sym^\bullet_{\O_X}(\bb L_X[1])$ and it is the unique such action when $X$ is a smooth scheme \cite[Theorem 1.3]{BenNad}.

Once one passes to derived stacks these statements are no longer true and we must take a formal completion of the loop space along the constant loops. This separation is essentially the incarnation from algebraic topology that a loop does not necessarily need to come from a single open set. After this adjustment all the statements hold for derived stacks. To sum up we make the following definition.
\begin{definition}
Let $X$ be a derived stack and $\widehat{\ca LX}$ the derived loop stack completed along the constant loops.
\begin{enumerate}
\item[(i)] A \emph{differential form} is an element of
$$
HH(X):=\O_{\widehat{\ca LX}}=Spf \lim_{n\to \infty}\left( Sym^\bullet_{\O_X}(\bb L_X[1])/Sym^{>n}_{\O_X}(\bb L_{X}[1])    \right),
$$
where the degree is given by the action of $\bb G_m$.
\item[(ii)] A differential form is \emph{closed} if it admits a lift to the negative cyclic homology
$$
HC^-(X):=\O_{\widehat{\ca LX}}^{B\bb G_a}
$$
\end{enumerate}
\end{definition}
Since the action of $\bb G_m$ and $B\bb G_a$ are compatible the degree of a closed differential form is well-defined. We will denote the degree preserving map by $i_X:HC^-(X)\to HH(X)$. In general, the notion of closedness is not well-defined as a differential form admits a space of closures, or \emph{keys}, which is the homotopy fiber of the map $i_X$ \cite{PTVV}. For our purposes it will be enough to pick a single lift.

Let us unravel what this means in terms of the tangent complex with the goal of generalizing the notion of a symplectic form to derived geometry. A $2$-form of degree $n$ is by definition a morphism 
$$
\omega^\vee:\ca O_X\to Sym^2_{\O_X}(\L_X[1])[-2+n],
$$
where the symmetric algebra is understood in the derived sense.  By the d\'ecalage isomorphism this is equivalent to a morphism 
$$
\omega^\vee:\O_X\to \wedge_{\O_X}^2\L_X[n].
$$
By the duality, $\L_{X}^\vee=\T_X$, this induces a morphism $\omega:\wedge^2_{\O_X}\T_X\to \O_X[n]$, and hence by adjunction a morphism $\Theta_\omega:\T_X\to \L_X[n]$. It is then clear how to generalize the notion of a symplectic form.

\begin{definition} \cite[Definition 1.16]{PTVV} An $n$-shifted symplectic form on a derived Artin stack $X$ is a closed $2$-form $\omega$ of degree $n$ such that the corresponding morphism
$$
\Theta_{\omega}:\bb T_X\to \bb L_X[n] 
$$
is an isomorphism in $D_{qcoh}(X)$.

\end{definition}

In the case when $X$ is a smooth scheme the above definition recovers the usual notion of a symplectic form. In particular there does not exist a degree $n$ symplectic form for $n\neq 0$ on a smooth scheme, i.e., $T_X\simeq \Omega_X[n]$ implies that $n=0$.




\section{Derived Marsden-Weinstein Quotient}
We now return to the setting of symplectic reduction. Recall that $(X,\omega_X)$ is a smooth symplectic variety with a Hamiltonian $G$-action, where $G$ is a reductive group, with moment map $\textbf J:X\to \frak g^*$.  In the case when $\mu$ is not a regular value of the moment map the canonical morphism $j:\textbf{J}^{-1}(\mu)\hookrightarrow X$ is not flat and $\bf J^{-1}(\mu)$ is not a smooth subvariety. The motto behind derived geometry is to consider a suitable derived enhancement of $\bf{J}^{-1}(\mu)$. To this end write $\textbf{J}^{-1}(\mu)$ as a fiber product coming from the diagram
$$
\xymatrix{\textbf{J}^{-1}(\mu):=X\times_{\frak g^*}\mu\ar[r]^-{p_{\mu}}\ar[d]_{p_X}&\mu\ar[d]\\
X\ar[r]^{\textbf{J}}&\frak g^*.}
$$
Therefore the derived enhancement of $\textbf{J}^{-1}(\mu)$ is naturally the homotopy fiber product, cf. Example \ref{fiberproduct},
$$
\ca X_{\mu}:=X\stackrel{h}{\times}_{\frak g^*}\mu.
$$
The orbit space, $\ca X_{red}:=[\ca X_{\mu}/G]$, is then a smooth derived Artin stack with a $G$-equivariant morphism $\pi:\ca X_{\mu}\to \ca X_{red}$.


By Example \ref{cotangentfiber} the tangent complex of $\ca X_{\mu}$ is 
$$
\bb T_{\ca X_\mu}=[T_X|_{\ca X_\mu}\to \frak g^*\otimes \O_{\ca X_\mu}].
$$
If $G$ acts on a derived scheme $X$ then the tangent complex of the derived quotient stack $[X/G]$ is given by the $G$-equivariant complex 
$$
\bb T_{[X/G]}=[\frak g \otimes \O_X\to T_X],
$$
where the map is given by differentiating the action of $G$ on $X$.
This implies that the tangent complex of the quotient $\ca X_{red}:=[\ca X_{\mu}/G]$ is 

\begin{equation}
\label{tangentcomplex}
\bb T_{\ca X_{red}}=[\frak g\otimes \O_{\ca X_\mu}\to T_X|_{\ca X_\mu}\to \frak g^*\otimes \O_{\ca X_{\mu}} ].
\end{equation}

The following theorem is then clear from the definitions.

\begin{theorem}
There exists a non-degenerate pairing $\omega_{red}:\bb T_{\ca X_{red}}\simeq \bb L_{\ca X_{red}}$ such that $j^*\omega_{X}=\pi^*\omega_{red}$.
\end{theorem}

\begin{proof}
By definition $\bb L_{\ca X_{red}}:=\bb T_{\ca X_{red}}^\vee$. The form $\omega_{red}$ is defined by the quasi-isomorphism of complexes
$$
\xymatrix{\frak g\otimes \O_{\ca X_{\mu}}\ar[r]\ar[d]&T_{X}|_{\ca X_{\mu}}\ar[r]\ar[d]^{\omega_{X}|_{\ca X_{\mu}}}&\frak g^*\otimes\O_{\ca X_{\mu}}\ar[d]\\
\frak g \otimes \O_{\ca X_{\mu}}\ar[r]&T^*_X|_{\ca X_\mu}\ar[r]&\frak g^*\otimes \O_{\ca X_{\mu}}.
}
$$
It is clear these complexes are quasi-isomorphic since $\omega_X$ is an isomorphism and $G$-equivariant. By construction $j^*\omega_X=\pi^*\omega_{red}$.
\end{proof}

What remains is to see that the form constructed is  closed in the derived sense i.e., $\omega_{red}\in HH_2(\ca X_{red})$ has a lift to $HC^-_2(\ca X_{red})$. We break the proof into two lemmas.

\begin{lemma}
In the notation from above the form $\pi^*\omega_{red}\in HH_2({\ca X_{\mu}})$ is $B\bb G_a$-equivariant.
\end{lemma}

\begin{proof}
A form is $B\bb G_a$-equivariant if and only if it admits a lift from $HH_*(X)$ to $HC_*^-(X)$ under the canonical map $i_X:HC^-_*(X)\to HH_*(X)$. By construction we see that $\pi^*\omega_{red}=j^*\omega_X$, therefore it is enough to show that $j^*\omega_X$ is $B\bb G_a$-equivariant. 

Let $\widetilde{\omega_X}$ be a lift of $\omega_X\in HH_2(X)$ to $HC_2^-(X)$, which exists since $\omega_X$ is closed.  If $j^{-*}$ denotes the induced map on negative cyclic homology then we claim that $j^{-*}(\widetilde{\omega_X})$ provides a lift of $j^*\omega_X$. Indeed, the commutative diagram
\begin{equation}
\label{CycHoch}
\xymatrix{HC^-_2(X)\ar[r]^{i_X}\ar[d]_{j^-*}&HH_2(X)\ar[d]^{j^*}\\
HC_2^-(\ca X_\mu)\ar[r]_{i_{\ca X_\mu}}&HH_2(\ca X_\mu)}
\end{equation}
implies that
$$
i_{\ca X_{\mu}}(j^{-*}(\widetilde{\omega_X}))=j^*i_X(\widetilde{\omega_X})=j^*\omega_X.
$$
\end{proof}

To finish the closedness part of the argument we need to show the following.
\begin{lemma}
If $\pi^*\omega_{red}$ is $B\bb G_a$-equivariant then $\omega_{red}$ is $B \bb G_a$-equivariant.
\end{lemma}

\begin{proof}
There is a commutative diagram
$$
\xymatrix{
HC_2(\ca X_\mu)^G\ar[r]\ar[d]&HH_2(\ca X_\mu)^G\ar[d]\\
HC_2(\ca X_{red})\ar[r]&HH_2(\ca X_{red})}
$$
The form $\pi^*\omega_{red}$ is $G$-equivariant and hence descends to the form $\omega_{red}\in HH_2(\ca X_{red})$. Now just repeat the argument from the previous lemma.
\end{proof}

It follows from a standard argument that symplectic reduction for points in $\frak g^*$ implies that symplectic reduction holds for any coadjoint orbit in $\frak g^*$ \cite{McDSal}. Indeed, let $\bb O$ be a coadjoint orbit and consider the symplectic variety $(X\times \bb O,\omega_{X}-\omega_{\bb O})$, where $\omega_{\bb O}$ is the standard symplectic form on the coadjoint orbit \cite{ChrGin}. It follows that $(X\times \bb O,\omega_{X}-\omega_{\bb O})$ has a Hamiltonian $G$-action with moment map
$$
\textbf {J}'(x,\eta)=\textbf {J}(x)-\eta,
$$
and it follows that 
$$
\textbf{J}^{-1}(\bb O)=\textbf{J}'^{-1}(0).
$$
The above argument implies the result for any coadjoint orbit. This completes the proof of Theorem \ref{derivedMW}.

\begin{remark}
\label{Compare}
The proof is very similar to the original proof of J. Marsden and A. Weinstein when $0$ is a regular value and the action of $G$ is free and proper. Their proof was essentially to first show that $T^*(\textbf{J}^{-1}(0)/G)=T^*(\textbf{J}^{-1}(0))/G$, cf. \eqref{tangentcomplex}, and then construct the symplectic form $\omega_{red}$. The closedness argument was not immediate but followed from injectivity of $\pi^*$ and 
$$
\pi^*d\omega_{red}=d\pi^*\omega_{red}=dj^*\omega_{X}=j^*d\omega_{X}=0.
$$
Lemma 3.2 and 3.3 provide the derived analogue of this argument.
\end{remark}

\begin{remark}
The same argument holds when the symplectic variety is replaced with a $0$-symplectic derived scheme since the cotangent complex is concentrated in degrees $[-N,0]$ for some positive integer $N$. However, the same argument does not hold for a $0$-symplectic derived Artin stack and it is not immediately clear how to generalize the argument to this case. 
\end{remark}


\bibliographystyle{plain}
\bibliography{SymplecticArtin}

\end{document}